\documentclass[11pt,reqno]{amsart}
\usepackage{fullpage}
\usepackage{mathrsfs}
\usepackage{amsfonts}
\usepackage{amsmath}
\usepackage{amsthm}
\usepackage{amssymb, stackrel}
\usepackage{latexsym}
\usepackage[all]{xy}
\usepackage{fullpage}

\usepackage[mathscr]{eucal}

\usepackage{xr-hyper}
\usepackage{hyperref}

\usepackage[dvipsnames]{xcolor}
\usepackage{graphicx}

\theoremstyle{plain}
\newtheorem{thm}[equation]{Theorem}
\newtheorem{cor}[equation]{Corollary}
\newtheorem{prop}[equation]{Proposition}
\newtheorem{lem}[equation]{Lemma}

\theoremstyle{definition}

\theoremstyle{remark}

\newtheorem{rem}[equation]{Remark}
\newtheorem{rems}[equation]{Remarks}

\numberwithin{equation}{section}

\newcommand{\mmod}{{\text{-}\mathrm{mod}}}

\newcommand{\hdot}{{\:\raisebox{2pt}{\text{\circle*{1.5}}}}}
\newcommand{\idot}{{\:\raisebox{2pt}{\text{\circle*{1.5}}}}}

\DeclareMathOperator{\CC}{{\mathrm{CC}}}
\DeclareMathOperator{\res}{{\mathrm{Res}}}
\DeclareMathOperator{\Res}{{\mathrm{DRes}}}
\DeclareMathOperator{\ind}{{\mathrm{Ind}}}

\DeclareMathOperator{\sym}{\mathrm{Sym}}

\DeclareMathOperator{\im}{\mathrm{Im}}
\DeclareMathOperator{\supp}{\mathrm{Supp}}
\DeclareMathOperator{\coh}{{\textit{Coh}}}
\DeclareMathOperator{\Ker}{\mathrm{Ker}}

\DeclareMathOperator{\End}{\mathrm{End}}

\DeclareMathOperator{\gr}{\mathrm{gr}}

\DeclareMathOperator{\SSS}{{\mathrm{SS}}}

\DeclareMathOperator{\Ad}{\mathrm{Ad}}
\DeclareMathOperator{\ad}{\mathrm{ad}}

\DeclareMathOperator{\Ind}{{\mathrm{DInd}}}

\DeclareMathOperator{\rad}{\mathsf{rad}}

\newcommand{\fp}{{\mathfrak p}}
\newcommand{\fq}{{\mathfrak q}}
\newcommand{\tfp}{{p}}
\newcommand{\tfq}{{q}}
\newcommand{\fm}{{\mathfrak M}}

\newcommand{\beq}{\begin{equation}\label}
\newcommand{\eeq}{\end{equation}}
\newcommand{\pf}{\begin{proof}}
\newcommand{\ep}{\end{proof}}

\DeclareMathOperator{\Spec}{\mathrm{Spec}}

\newcommand{\iso}{{\;\stackrel{_\sim}{\to}\;}}

\DeclareMathOperator{\Hom}{\mathrm{Hom}}

\def\ccirc{{{}_{\,{}^{^\circ}}}}

\newcommand{\bdot}{{\boldsymbol{\cdot}}}

\renewcommand{\o}{\otimes }

\newcommand{\dash}{\text{-}}

\renewcommand{\t}{{\mathfrak t}}

\newcommand{\ce}{{\mathcal E}}

\newcommand{\GT}{{G\times T}}

\newcommand{\wt}{\widetilde }

\newcommand{\Id}{{\operatorname{Id}}}
\newcommand{\zz}{{\mathcal Z}}
\newcommand{\xx}{{\mathcal X}}

\newcommand{\fc}{{\mathfrak c}}

\newcommand{\eq}{{\ =\ }}

\newcommand{\pq}{{p\times q}}

\newcommand{\ccong}{\ \cong\ }

\newcommand{\gt}{{G\times T}}

\newcommand{\erem}{\hfill$\lozenge$\end{rem}\vskip 3pt }
\newcommand{\erems}{\hfill$\lozenge$\end{rems}\vskip 3pt }
\newcommand{\fy}{{Y}}

\newcommand{\scr}[1]{\mathscr{#1}}
\newcommand{\dcoh}{{D^b_{\textit{coh}}}}

\newcommand{\lo}{{\,\stackrel{L}\o\,}}

\newcommand{\xxx}{{X\times T}}

\newcommand{\cg}{{\mathcal G}}

\newcommand{\ddx}{{\dd(X)}}

\newcommand{\nn}{{\mathbf N}}

\newcommand{\cmm}{{\boldsymbol{\mathcal M}}}

\newcommand{\la}{\lambda}

\renewcommand{\gg}{{G\times \g}}

\newcommand{\dd}{{{\mathscr{D}}}}
\newcommand{\oo}{{\mathcal{O}}}
\newcommand{\kk}{{\mathcal{K}}}

\newcommand{\vp}{{\varpi}}
\newcommand{\ug}{{{\mathcal{U}}{\mathfrak{g}}}}

\renewcommand{\tt}{{T\times\t}}

\newcommand{\pa}{\partial }
\renewcommand{\H}{{\scr H}}

\newcommand{\ddg}{{\dd(G)}}
\newcommand{\ddt}{{\dd(T)}}
\newcommand{\ddtw}{{\dd(T)^W }}

\newcommand{\tg}{{\wt G}}

\newcommand{\al}{\alpha }
\newcommand{\ut}{{\mathcal{U}\t}}

\newcommand{\hc}{{Harish-Chandra }}

\newcommand{\C}{\mathbb{C}}
\newcommand{\g}{{\mathfrak{g}}}

\renewcommand{\b}{\mathfrak{b}}

\newcommand{\bb}{{\scr B}}

\newcommand{\mm}{{\mathbf M}}

\newcommand{\T}{{\mathcal T}}
\newcommand{\inv}{^{-1}}

\newcommand{\cf}{{\mathcal F}}
\newcommand{\Z}{{\mathbb Z}}
\newcommand{\en}{{\enspace}}

\newcommand{\vi}{${\en\sf {(i)}}\;$}
\newcommand{\vii}{${\;\sf {(ii)}}\;$}
\newcommand{\viii}{${\sf {(iii)}}\;$}
\newcommand{\iv}{${\sf {(iv)}}\;$}
\newcommand{\vv}{${\sf {(v)}}\;$}
\newcommand{\sseteq}{\subseteq }
\newcommand{\sset}{\subset}
\newcommand{\sminus}{\smallsetminus}
\newcommand{\into}{\,\hookrightarrow\,}

\newcommand{\mto}{\mapsto}
\newcommand{\onto}{\,\twoheadrightarrow\,}

\newcommand{\Ga}{\Gamma }

\newcommand{\om}{\omega }

\title{Parabolic induction and
 the Harish-Chandra $\boldsymbol{\mathcal D}$-module}

\author{Victor Ginzburg}\address{
Department of Mathematics, University of Chicago,  Chicago, IL 
60637, USA.}
\email{ginzburg@math.uchicago.edu}

\begin{document}
\maketitle
\begin{flushright} {\em To the memory of Tom Nevins}
\end{flushright}
\bigskip

\begin{abstract} Let $G$ be  a  reductive  group 
and $L$ a
Levi subgroup.  Parabolic
induction and  restriction
are a pair  of adjoint functors
between  $\Ad$-equivariant derived  categories of 
either constructible sheaves or  (not necessarily holonomic)
$\dd$-modules on $G$ and $L$, respectively.
Bezrukavnikov and Yom Din proved,
generalizing a classic result of Lusztig, 
that these functors  are exact. 
In this paper, we consider 
a special case where $L=T$ is a maximal torus.
We give  explicit  formulas for  parabolic induction and restriction  in
terms of   the
Harish-Chandra $\dd$-module on
$\gt$.
We show that  this
 module is flat  over $\dd(T)$, which easily implies
 that  parabolic induction and restriction are
 exact   functors between the corresponding abelian categories of $\dd$-modules.
 \end{abstract}
\maketitle

{\small
\tableofcontents
}


\section{Main results}\label{main-sec}

\subsection{} Throughout the paper, we fix
a connected and simply-connected complex semisimple group $G$.
Let $T$ be the abstract Cartan  torus,
$W$ the  Weyl group, $R_+$  the  set
of positive roots, and  $\rho$ the half-sum of positive roots. Let
\beq{del}
\delta :=\prod\nolimits_{\al\in R_+}\ {(e^{-\al}-1)}\, \in\C[T],
\eeq
and $T_r:=\{t\in T\mid \delta(t)\neq0\}$, a $W$-stable Zariski open subset of $T$.

Let  $\dd_X$ denote the sheaf of differential operators on 
a smooth complex algebraic variety $X$ and
 $\ddx=\Ga(X,\dd_X)$.
 Let  $\C[G]^G\sset \C[G]$ and  $\dd(G)^G\sset \dd(G)$ be the subalgebras of
 $\Ad G$-invariant regular functions and  differential operators on $G$,
 respectively.

 Throughout the paper, we let $W$ act on the algebra $\ddt$ via the
 `dot-action' 
 $w: u\mto w\bdot u$, rather than the natural action $u\mto w^*u$.
 By definition, the  differential operator
 $w\bdot u$ acts on  $f\in \C[T]$   by
 the formula $w\bdot u: f\mto e^{-\rho}\cdot (w^*u)(e^\rho\cdot f)$.
Let $\ddtw\sset \ddt$ be the algebra
 of invariants of the dot-action.
 Note that for zero order differential operators, one has $w\bdot u=w^*u$.
 Thus, the natural algebra imbedding $\C[T]\into \ddt$ respects the $W$-actions

 \rule{4cm}{0.4pt}

 {\footnotesize{\noindent I am  grateful  to Tsao-Hsien Chen for useful discussions
     and comments that
     have led to a number of improvements of the original text.}}
 
\newpage

\noindent
 so, we have $\C[T]^W=\C[T]\cap \ddtw$.

 \hc   extended  the Chevalley
 isomorphism $\C[G]^G\iso \C[T]^W$, to be written as $f\mto f_T$,
 to a `radial parts' homomorphism $r: \dd(G)^G\to\dd(T)^W$
 such that the following equation holds,
 cf. \cite[Section 6]{HC},
  \beq{rad-f} r(u)(\delta\cdot f_T)\eq \delta\cdot (u(f))_T,\qquad\forall u\in\ddg^G,\
 f\in \C[G]^G.
 \eeq
 
It is clear from the equation that   $r(f)=f_T$ for all
$f\in \C[G]^G\subset \ddg^G$.

\begin{rem}\label{rad rem} Equation \eqref{rad-f} may be used as a definition
  of $r(u)$, viewed as a differential operator  on $T_r$.  It is not at all clear,
  however,
  that the  differential operator thus defined has no poles at the divisor $\delta=0$.
  \erem

Let $\g$ and $\t$ be the Lie algebras of $G$ and $T$, respectively.
 Let $\ug$ and $\ut$  denote the corresponding enveloping algebras.
 The algebra $\ut\cong\sym\t$ is naturally identified
 with  the algebra of  translation invariant differential
operators on $T$. Similarly, the center  $Z\g$, of $\ug$, may (and will) be
identified with the algebra
bi-invariant differential
operators on $G$. With these identifications,
the
 restriction  of the radial parts map  to  the subalgebra  $Z\g\sset \ddg^G$
reduces to
the  \hc isomorphism $Z\g\iso (\sym\t)^W$.

The  differential of  the action of $G$
on itself by conjugation sends an element $a\in\g$
to a vector field $\ad a$, on $G$. This gives a map
  $\ad: \g\to\dd(G)$. Let
$\dd(G)\ad\g$ be a left ideal of 
the algebra $\dd(G)$ generated by  the image of the map $\ad$.
Thus, $\nn:=\ \dd(G)/\dd(G)\ad\g$ is
a left $\ddg$-module.\footnote{This $\ddg$-module  has been considered earlier
 by M. Kashiwara \cite{Ka}.}
The algebra $(\End_\ddg \nn)^{op}$ can be   naturally  identified
with $(\ddg/\ddg\ad\g)^G$, a quantum Hamiltonian reduction of
the algebra $\ddg$ with respect to the $\Ad G$-action.
Further, it is immediate from \eqref{rad-f} that the 
map  $r$  factors through a map
$\big(\dd(G)/\dd(G)\ad\g\big)^G\to\ddtw$, furthermore, the resulting map
\beq{ag}\rad:\, (\End_\ddg \nn)^{op}\,=\,\big(\dd(G)/\dd(G)\ad\g\big)^G\,
\xrightarrow\,\ddtw,
 \eeq
is  an algebra  homomorphism.
 An  important result  due to 
 Levasseur and Stafford \cite{LS1}-\cite{LS2}, cf. also
 Wallach  \cite{Wa}, says that the  map \eqref{ag}  is, in fact, an   isomorphism.

One can view $\nn$ as a bimodule with respect
 to the left action of the algebra $\ddg$ and the natural right 
  action of the algebra
 $(\End_\ddg \nn)^{op}$. 
 Therefore,  transporting the right action  via the map $\rad$ gives
  $\nn$ 
  the structure of a $(\ddg,\ddtw)$-bimodule.

  \begin{rem}\label{ls}    The  surjectivity of the map $\rad$ 
     follows from
    the fact, \cite{LS1}, Lemma 9,  that the algebra $\dd(T)^W$ is
 generated by the subalgebras $\C[T]^W$ and $(\sym\t)^W$.
 Proving injectivity is much more difficult.
 \end{rem}

 \subsection{} \label{s not}
  Given a morphism $f: X\to Y$ of smooth varieties, we follow the notation of
  \cite{HTT}, ch.1, and write 
  $f^*(\dash)=\oo_X\lo_{f^\hdot\oo_Y}\, f^\hdot (\dash)$,
  resp.  $\int_f(\dash)=Rf_\idot(\dd_{Y\leftarrow X}\lo_{\dd_X}\, \dash)$, 
for  derived  pull-back,  resp. push-forward, functors on $\dd$-modules.
Let $\int^i_f \ce=\H^i(\int_f\ce)$ denote the $i$th cohomology $\dd_Y$-module.

Let $\coh(\dd_X)$
 be the abelian category of
 coherent $\dd_X$-modules and
  $\dcoh(\dd_X)$  a  full subcategory of the bounded derived category of
 $\dd_X$-modules whose objects $\ce$ have coherent cohomology $\H^\hdot(\ce)$.
 Thus,  $\coh(\dd_X)$ is the heart of the  natural t-structure
 on  $\dcoh(\dd_X)$.
 We also consider  
 the  abelian category  $\coh(\ddx)$ of finitely generated $\ddx$-modules,
 and the corresponding 
  bounded derived category  $\dcoh(\ddx)$.

Let $\bb$ be the flag variety, the variety
of all Borel subgroups $B\sset G$.
We have a diagram
\[
  \xymatrix{
    G\ &&\ 
    \tg:=\{(B,g)\in\bb\times G\mid g\in B\}\  \ar[ll]_<>(0.5){p}\ar[rr]^<>(0.5){q} &&\
    T.
  }
  \]
Here, $p$ is the second projection  
(the Grothendieck-Springer morphism) and the map $q$ is defined 
by the assignment
$(B,g)\, \mapsto g\,{\textrm{mod}} [B,B]$.

The functor of {\em parabolic induction}
is defined by the formula $\int_p   q^*: \dcoh(\dd_T)\to\dcoh(\dd_G)$.

The main result of the paper reads

\begin{thm}\label{main-thm} \vi There is an  isomorphism of functors that makes
  the following diagram commute
\[
\xymatrix{
\dcoh(\dd_T)\ \ar@{=}[d]^<>(0.5){R\Ga(T,\dash)}
\ar@<0.4ex>[rrr]^<>(0.5){\int_p   q^*(\dash)}
&&&\ \dcoh(\dd_G),\ \ar@{=}[d]^<>(0.5){R\Ga(G,\dash)}\\
\dcoh(\ddt)\ \ar@<0.4ex>[rrr]^<>(0.5){\nn\lo_\ddtw(\dash)}
&&&\ \dcoh(\ddg)
}
\]

\vii  The $(\dd(G),\dd(T)^W)$-bimodule $\nn$ is simple.
Furthermore, $\nn$ is  flat as a right
$\ddt^W$-module; hence, the functor $\int_p   q^*$ is t-exact
(with respect to the natural t-structure).
\end{thm}

An important class of  $\dd_T$-modules comes from local systems.
Specifically, associated with a $\sym\t$-module $V$ one has a
$\dd_T$-module $\dd_T\o_{\sym \t} V$, which is isomorphic
to  $\oo_T\o V$ as an  $\oo_T$-module.
The $\dd_T$-action gives a flat connection on  $\oo_T\o V$;
the sheaf of {\em holomorphic} horizontal sections
of that connection is a local system on $T$.

From Theorem \ref{main-thm} we will deduce (see Section \ref{d sub}) the following result.

\begin{cor}\label{mla} Let  $\la\in \t^*=\Spec(\sym\t)$ be a regular point of 
  the dot-action of $W$, e.g. $\la=0$. Then, for
  any finite dimensional $\sym\t$-module $V$
  such that $\supp V=\{\la\}$,  there is an isomorphism
  \begin{align*}
      \Ga\big(G,\,\mbox{$\int^0_p$}\,  q^*(\dd_T\o_{\sym \t} V)\big)\ccong
      \nn\o_{(\sym\t)^W} V,
        \end{align*}
of $\ddg$-modules; furthermore, $\int^i_p  q^*(\dd_T\o_{\sym \t} V)=0$  for all $i\neq 0$.
      \end{cor}

      Let $I_\la\subset \sym\t$ be the maximal ideal of  a regular point $\la\in\t^*$
      and identify  $Z\g$ with $(\sym\t)^W$.
         In the special case where 
      $V=S/I_\la$                      the corollary says that there is an isomorphism
      \beq{hk iso} \int_p  q^*(\dd_T/\dd_T I_\la)\ccong
        \dd_G\big/\big(\dd_G\ad\g + \dd_G\cdot (Z\g\cap I_\la)\big).
        \eeq

\begin{rems}\label{byk} \vi    Bezrukavnikov  and Yom ~Din  \cite{BY}
  have shown that, in the general case of an arbitrary parabolic,
  the functor of   parabolic induction
  (as well as  parabolic restriction, cf.
Section \ref{pf-sec})  is an exact functor between equivariant derived categories.
The methods of \cite{BY}  are
different from ours; in particular,  the arguments in {\em loc cit} 
depend on  the  `second adjointness'  theorem of Drinfeld and Gaitsgory \cite{DG}
and on  results of Raskin \cite{Ra} on holonomic defects of $\dd$-modules.
\vskip2pt

\vii The flatness statement in Theorem \ref{main-thm}(ii)
is somewhat surprising since 
a natural commutative counterpart of this statement  is false,
cf. Remark \ref{z flat}(ii).
A weaker flatness result 
has been established earlier by Kashiwara
\cite{Ka}.  Specifically,   Kashiwara
proved  that $\nn$ is flat as a right $Z\g$-module, where
$Z\g$ is identified with the commutative subalgebra  $(\sym\t)^W\sset\ddt^W$.
\vskip2pt

\viii     An analogue  of  isomorphism \eqref{hk iso}  in the setting of
$\dd$-modules  on the Lie algebra $\g$ is one of the main results 
 of Hotta and Kashiwara \cite{HK1}.
 Levasseur and Stafford \cite{LS2}
 proved that $\dd(\g)/\dd(\g)\ad\g$ is flat as a right
 $(\sym\g)^G$-module, an analogue   of the flatness result  from \cite{Ka}
 in the  Lie algebra  setting.

 \iv Some constructions closely related to Theorem \ref{main-thm} have been considered
earlier by  D. ~Ben-Zvi and S. Gunningham,  \cite{BZG}.

\vv  Harish-Chandra's original construction
    of the radial parts map $r$  can be translated into modern language 
    as follows.
        Fix a Borel subgroup $B\sset G$ with Lie algebra
    $\b$ and let  $B$ act on $G$ by conjugation and act  on $T$ trivially.
We have a diagram $T=B/[B,B]\xleftarrow{\al} B
\xrightarrow{\beta} G$, of  $B$-equivariant maps.
The space ${L}=\Ga(G, \int_\beta \al^*\dd_T)$ has the natural structure of a
$(\ddg,\ddt)$-bimodule.  Furthermore, 
$\int_\beta \al^*\dd_T$ is strongly equivariant as a left $\dd_G$-module
and the resulting  $B$-action on $L$
 commutes with the  right $\ddt$-action.
    Let $\gamma=2\rho$ be the sum of positive roots
        viewed as a character  of $B$,
         and ${L}_{\gamma}\subset L$
        the ${\gamma}$-weight space of the $B$-action.

         The bimodule $L$ comes equipped with a  canonical section $s_{L}=(dg)\inv db$,
        where we use the notation
        $dk$ for a left invariant volume form on an algebraic group $K$.
      The section $s_{L}$ has weight ${\gamma}$,
    hence the right action of $\ddt$ on ${L}$ gives a map
    $\ddt\to {L}_{\gamma},\, D\mto s_{L} D$.
        Using an associated graded of ${L}$ with respect to a natural good filtration,
    it is not difficult to show
that   this map is actually a bijection.
    It follows that the image  of the map $D\mto s_{L} D$ is  $\ddg^G$-stable.
    We deduce that for every $u\in\ddg^G$ there is a unique
    differential operator $r(u)\in \ddt$
    such that, inside $L$, one has
    $u s_{L}=s_{L}\,r(u)$.
    It is immediate 
    that the assignment $u\mto r(u)$
    gives an algebra homomorphism $r: \ddg^G\to \ddt$.
        From the $G$-invariance of $u$ one
    deduces  that $r(u)$ is a $W$-invariant
    operator and   formula \eqref{rad-f} holds. 
  \end{rems}

\section{The \hc $\dd$-module}\label{hc-sec} 
A key role in our approach to parabolic induction  is played by
the \hc  module, a   left  $\dd_{\gt}$-module
defined as follows:
\beq{cmm}\cmm:=
\dd_\gt\big/\big(\dd_\gt\,(\ad\g\o 1)+\dd_\gt\,\{u\o1-1\o \rad(u),\ u\in \dd(G)^G\}\big).
\eeq

The following  result,  essentially due to Hotta and Kashiwara
\cite{HK1}, provides a geometric interpretation of the \hc module in terms of the diagram 
 $G\xleftarrow{p}\wt G \xrightarrow{q} ~T$.

\begin{thm}\label{pimm} There is an isomorphism
  $\int^0_{\pq}\oo_\tg\ccong\cmm$;
  furthermore, $\int^i _{\pq}\oo_\tg=0$ for all $i\neq 0$.
\end{thm}

\begin{rem}
In   \cite{HK1}, the authors considered a Lie algebra version of the above theorem, 
cf. \cite{HK2} and  \cite{Ka} for closely related results concerning $\int_p \oo_\tg$.
The strategy of the proof of  Theorem \ref{pimm} outlined below
is similar to the strategy used in   \cite{HK1}.
There are, however, two differences. First, 
the arguments in  \cite{HK1} involve Fourier transform,
which is not available
in the group setting.
The second difference is that the definition
of the Lie algebra counterpart of the \hc module  given in \cite{HK1}
doesn't quite match formula \eqref{cmm}. 
Specifically, the   elements of the form
$u\o1-1\o \rad(u)$ that appeared in  the definition   in {\em loc cit} are
the ones where $u$ is taken to be  either an element of  the subalgebra
$\C[\g]^G\sset \dd(\g)^G$ or of the subalgebra 
$(\sym\g)^G\sset \dd(\g)^G$,
 of $G$-invariant differential operators  with {\em constant} coefficients.
  The difference between the (Lie algebra analogue of)
 \eqref{cmm} and the formula in  \cite{HK1}  doesn't affect the resulting $\dd$-module,
 thanks to \cite[Lemma 9]{LS1}, cf. Remark \ref{ls}.
\erem

The rest of this subsection is devoted to the proof of Theorem \ref{pimm}.

Let $G_{rs}\sset G$ be the regular semisimple locus of $G$.
 Let  $\fc=\Spec\C[G]^G$. We consider  a fiber product $G\times_{\fc} T$ and its open subset
 $\fy:=G_{rs}\times_{\fc} T_r$, which is a smooth  connected
 closed subvariety of $G_{rs}\times T_r$.
Let $N_\fy\sset \T^*(G_{rs}\times T_r)$ be (the total space of)  the conormal bundle on
this subvariety,
$\overline{N_\fy}$ the closure of $N_\fy$ in  $\T^*(G\times T)$,
and $[\overline{N_\fy}]$  the fundamental class of $\overline{N_\fy}$.

Let $\SSS(-)$, resp.  $\CC(-)$, denote the characteristic variety,
resp. characteristic cycle, of a $\dd$-module.

 \begin{prop}\label{cc} One has $\CC(\cmm)=[\overline{N_\fy}]$.
   Further, 
   for all $i$, we have  $\SSS\big(\int^i_{p\times q}\oo_\tg\big)\subseteq\overline{N_\fy}$.
     \end{prop}

The proof of the proposition is based on Lemma \ref{simple} below that provides
  a description
of the    variety  $\overline{N_\fy}$ 
in terms of the commuting scheme $\zz$,
a closed {\em not} necessarily reduced
subscheme  of $\gg$ defined by the equations 
\[\zz=\{(g,x)\in G\times \g\mid \Ad g(x)=x\}.
\]

A choice of  imbedding $T\into G$ gives an imbedding $\tt\into\zz$.
We let  $\xx$ be  a closed
subscheme of $\zz\times\tt$   
 defined  by the equations
\[
  \xx=\{(g,x,t,h)\in \zz\times\tt\mid f(g,x)=f|_{\tt}(t,h),\quad
  \forall f\in \C[\zz]^G\}.
    \]
      The scheme $\xx$ is independent of the choice of
   an imbedding $T\into G$.

To simplify notation, we 
identify $\g$ with $\g^*$, resp. $\t$ with $\t^*$,
using the Killing form. This gives
natural identifications $\gg=G\times \g^*=\T^*G$, resp.
$\tt=T\times \t^*=\T^*T$,
where  $\T^*X$ denotes the total space of the cotangent bundle on
a smooth variety $X$.
Thus, we may (and will) view $\xx$ as a closed subscheme of $\T^*(\gt)$.

The proof of the following lemma  is similar to the proof of 
 \cite[Lemma 4.2.1]{HK1}, cf. also  ~\cite[Lemma 1.5]{Gi}.

\begin{lem}\label{simple} The variety $N_\fy$ is a Zariski open and dense
  subscheme of $\xx$; explicitly,
  we have
  \[N_\fy\,=\,\xx\,\cap\, \T^*(G_{rs}\times T_r).
  \qedhere  \]
\end{lem}

\begin{rem} 
   It is known that the commuting variety $\zz$  is an irreducible and generically reduced scheme
of dimension $\dim G+\dim T$,
\cite{Ri}. The  scheme $\xx$ is not reduced even in the case $G=SL_2$.
It follows from Lemma \ref{simple}  that the scheme  $\xx$   is  generically reduced and
   irreducible; furthermore,   
   the reduced scheme $\xx_{\text{red}}$,
   called the {\em isospectral commuting variety},  equals $\overline{N_\fy}$.
\end{rem}

\begin{proof}[Proof of Proposition \ref{cc}]
  We write formula \eqref{cmm} in the form
  $\cmm=\dd_{G\times T}/{\mathcal J}$, where ${\mathcal J}$ is a left
  ideal of $\dd_{G\times T}$. 
The standard  ascending filtration on $\dd_{\GT}$
by  order of the
differential operator induces  a  quotient filtration
on $\cmm$ such that $\gr\cmm=\oo_{\T^*(G\times T)}/\gr{\mathcal J}$.
It is immediate from \eqref{cmm} that the ideal
$\gr{\mathcal J}$ contains the ideal of definition of the scheme $\xx$,
cf. \cite[Lemma 2.4.3]{Gi} for  a Lie algebra counterpart.
Hence, there is a surjective morphism $\oo_{\xx}\onto \gr\cmm$.
By Lemma \ref{simple}, we know  that $N_Y$ is  open dense in $\xx$ and, moreover,
the scheme $\xx$ is reduced at every closed point of $N_Y$.
This implies the equation $\CC(\cmm)=[\overline{N_Y}]$,
since $\overline{N_Y}$ is an irreducible Lagrangian subvariety.

To prove the second statement of Proposition \ref{cc}
we use a well known  upper bound on  the characteristic variety of a proper direct image, \cite[Section 2.5]{HTT}.  By a  straighforward calculation, one finds that
this upper bound forces 
$\SSS\big(\int^i_{p\times q}\oo_\tg\big)$ be set-theoretically contained in
 $\xx_{\text{red}}$, that is, in $\overline{N_Y}$.
 \end{proof}

Let $\kk_X$ denote the canonical sheaf
of a smooth variety $X$.
We choose and fix  translation invariant global sections $dt$ and $dg$ of $\kk_T$ and $\kk_G$, respectively.

One  shows by computing Jacobians 
that there is a nowhere vanishing $G$-invariant global section
$\om$ of $\kk_\tg$ such that ${p}^*(dg)={q}^*\delta\cdot\om$,
cf.
\cite[(4.1.4)]{HK1}   for a Lie algebra analogue.

\begin{proof}[Sketch of Proof of Theorem \ref{pimm}]
  Let $\tg_{rs}:=p\inv(G_{rs})=q\inv(T_r)$.
The map $p\times q$ restricts to an isomorphism $\wt G_{rs}\iso \fy$.
It follows that $\big(\int^i_{p\times q}\oo_\tg\big)|_{G_{rs}\times T_r}\cong \int^i_\Delta \oo_\fy$,
where  $\Delta $ is  the closed imbedding
$\fy\into G_{rs}\times T_r$.

If  $i\neq 0$
we deduce that
$\big(\int^i_{p\times q}\oo_\tg\big)|_{G_{rs}\times T_r}=\int^i_\Delta \oo_\fy=0$.
By Proposition \ref{cc}, this  forces $\SS\big(\int^i_{p\times q}\oo_\tg\big)$ be
contained in
$\overline{N_Y}\sminus N_Y$, which implies
 that $\int^i_{p\times q}\oo_\tg=0$   for all $i\neq 0$.

Similarly,
in the case $i=0$,  we find that the characteristic cycle of
$\big(\int^0_{p\times q}\oo_\tg\big)|_{G_{rs}\times T_r}$ is equal
to $\CC(\int^0_\Delta \oo_\fy)=[N_Y]$.
Using this equation and the inclusion $\SS(\int^0_{p\times q}\oo_\tg)\sset \overline{N_Y}$,
cf. Proposition \ref{cc}, we deduce that
$\CC(\int^0_\Delta \oo_\fy)=[\overline{N_Y}]$.
It follows that $\int^0_\Delta \oo_\fy$ is a simple $\dd$-module;
in paricular,  this $\dd$-module has no nonzero sections supported on the
 complement of the open set $G_{rs}\times T_r$.

Observe next that the $\dd_\gt$-module
$\int^0_{p\times q}\oo_\tg$ comes equipped with
 a  canonical global section  $s=(dg\, dt)\inv\,\o\,({p}\times{q})_*\om$.
 For any $u\in \ddg^G$, using equation  \eqref{rad-f} one shows
 that the section $(u\o1-1\o \rad(u))s$ vanishes on the open
 set $G_{rs}\times T_r\sset \gt$.
 Hence, this section is identically zero, by the previous paragraph.
 Further,  
 the section $s$ is $G$-invariant, hence it
is  annihilated  by  the elements
$\ad a\o 1\in\dd(\gt)$, for all $a\in\g$.
We conclude that  the map $\dd_\gt\to \int^0_{p\times q}\,\oo_\tg,\ u\mto u(s)$,
descends to a well-defined $\dd$-module morphism $F: \cmm\to \int^0_{p\times q}\,\oo_\tg$.
It is straightforward to check  that the
composition  $\cmm|_{G_{rs}\times T_r}\xrightarrow{F|_{G_{rs}\times T_r}}
(\int^0_{p\times q}\,\oo_\tg)\big|_{G_{rs}\times T_r}\iso\int^0_\Delta\oo_Y$
is an
isomorphism of $\dd$-modules on $G_{rs}\times T_r$.
Hence, the $\dd$-module  $\Ker(F)$, resp. $\im(F)$, is supported
on the complement of  $G_{rs}\times T_r$.
Since the characteristic variety of  $\cmm$, resp. $\int^0_{p\times q}\,\oo_\tg$, is
  contained
  in $\overline{N_\fy}$, this forces $\SS(\Ker(F))$, resp. $\SS(\im(F))$, be
  contained in $\overline{N_Y}\sminus N_Y$.
  We conclude that $\Ker(F)=0$ and $\im(F)=0$.
Thus, the map  $F$ yields an isomorphism  $\cmm\iso \int^0_{p\times q}\,\oo_\tg$.
\end{proof}

\begin{rems}\label{z flat}
\vi   The  map 
$p\times q: \tg\to G\times_\fc T$ is well known to be  a small morphism. 
This yields an alternative proof of the  fact that
$\int_{\pq} \oo_\tg$  is a simple  holonomic $\dd_{\gt}$-module.
\vskip3pt

\vii Sending a commuting pair $(g,x)\in \zz$ to the pair  $(g_{ss}, x_{ss})$
 of the corresponding semisimple components of the Jordan decomposition,
gives a  map $\pi: \zz\to
 (\tt)/W$.
 This map is
 not flat (and neither is the map $\zz_{\text{red}}\to(\tt)/W$).
 Indeed, the general fiber of $\pi$
has dimension $\dim G-\dim T$, while  the fiber over the $W$-orbit of the point
$(1,0)\in\tt$ has dimension $\dim G$.

 It is easy to show that 
$\SSS(\nn)=\zz$. Furthermore, the  algebra map  $\pi^*:
\C[\tt]^W\into \C[\zz]$, induced by $\pi$, may be viewed as a commutative analogue
of the composite map $\ddtw\cong (\ddg/\ddg\ad\g)^G\into
\ddg/\ddg\ad\g=\nn$.
Thus, one may view the $\C[\tt]^W$-module $\C[\zz]$
as a  commutative analogue
of the $\ddtw$-module $\nn$.
We know that $\nn$ is flat over $\ddtw$, by Theorem \ref{main-thm}.
However, 
$\C[\zz]$ is not flat over $\C[\tt]^W$ since the map $\pi$ is not flat.
 \end{rems}

\section{Proof of Theorem \ref{main-thm}}\label{ind-sec} 
\subsection{}\label{xyz sec}
Below, we make no distinction between left and right $\dd_T$-modules;
specifically, we identify a left $\dd_T$-module $\cf$ with
the  right $\dd_T$-module $ \kk_T\o_{\oo_T} \cf$ using
the map
$\cf\mto \kk_T\o_{\oo_T} \cf$, $f\mto dt \o f$.
The left  and right actions are related by the formula
$(dt\o f)u=dt\o u^t(f)$, where $u\mto u^t$
is an anti-involution $\dd_T\to \dd_T$ which sends
$\frac{\pa}{\pa t}$ to $-\frac{\pa}{\pa t}$
and restricts to the identity  on $\oo_T\sset \dd_T$.

More generally,  let $X$ be a   smooth variety 
and write  $\fp: X\times T\to X$ for
 the first, resp.  $\fq: X\times T\to T$ for
 the  second, projection.
Then, we can (and will)  
identify left  $\dd_{X\times T}$-modules
with $(\fp^*\dd_X, \fq^*\dd_T)$-bimodules
using a construction similar to the one above.

Now, fix   a smooth variety $Z$
and  a pair of morphisms $p,q$, as depicted in
the
diagram
\[
\xymatrix{
 &&&\ Z\ \ar[dlll]_<>(0.5){ p} \ar[drrr]^<>(0.5){q} \ar[d]^<>(0.6){\pq}&&&\\
X\ &&&\  X\times T\ \ar[lll]_<>(0.3){\fp}\ar[rrr]^<>(0.3){\fq} &&& T.
}
\]

Let  $\fm:=\int_{\pq}\oo_Z$.  The proof of Theorem \ref{main-thm} is based on the following general result.

\begin{lem}\label{xyz} 
\begin{enumerate}

\item  There are   isomorphisms of functors
\beq{iso fun}
\int_{p} {q}^* (\dash)   \ccong
   \int_\fp(\fm\,\lo_{\fq^\hdot\oo_T}\,\fq^\hdot(\dash))\ccong
   R\fp_\idot(\fm\lo_{\fq^*\dd_T} \,\fq^*(\dash)).
\eeq


\item If the morphism $ p$ is proper
  and the morphism $ q$ is
  smooth, then  $\int^i_{\pq}\oo_Z$
 is a flat $\fq^\hdot\oo_T$-module, for all $i$.
\end{enumerate}
\end{lem}

\begin{proof} 
We will abuse notation and write $\o$ for $\lo$, resp. $\dd_T$  for $\fq^*\dd_T$.
Let $\C$ denote the trivial 1-dimensional representation of $\ut$.
Recall that tensoring over $\oo$ gives  a tensor product operation on
$\dd$-modules.

For any left $\dd_{X\times T}$-module  $\cg$, 
we have a chain of
natural isomorphisms:
\begin{align*}
    \big(\kk_T\o_{\oo_T}  (\fm\o_{\oo_{X\times T}}\cg)\big)\,\o_{\dd_T}\,\oo_T\ &\xleftarrow{\cong}\
\big(\C dt\o (\fm\o_{\oo_{X\times T}} \cg)\big) \,\o_{\ut}\, \C\\
   &\xrightarrow{\cong}\ (\C dt\o \fm)\,\o_{\ut}\ \cg \ \xrightarrow{\cong}\ (\kk_T\o\fm)\,\o_{\dd_T}\,\cg.
    \end{align*}
    Here, the first and the  third isomorphism
    are
    induced  by the natural imbeddings $\C dt\into \kk_T$ and $\C\into \oo_T$;
    the second isomorphism sends $dt\o (u\o u')\o 1$ to $(dt\o u)\o u'$.

    We  identify the left $\dd_{X\times T}$-module $\fm$
    with $\kk_T\o\fm$, a $(p^*\dd_X, q^*\dd_T)$-bimodule.
    With this identification, the composite   isomorphism above reads as follows
    $    (\fm\o_{\oo_{X\times T}}\cg)\o_{\dd_T}\oo_T\ccong
    \fm\,\o_{\dd_T}\,\cg$. Thus, by the definition of the functor $\int_\fp$,
we have
\beq{f1}
\int_{\fp}(\fm\,\o_{\oo_{X\times T}}\, \cg)\eq
\int_\fp\big((\fm\o_{\oo_{X\times T}}\cg)\o_{\dd_T}\oo_T\big)\eq
R\fp_\idot\big(\fm\,\o_{\dd_T}\,\cg).
    \eeq

    Next, let $\cf$ be a $\dd_T$-module. Writing  $\vp:=\pq$, and using  the
projection formula, cf. \cite{HTT}, Corollary 1.7.5,
we compute
    \begin{align}
  \mbox{$\int_{p}$} {q}^*  \cf \eq\mbox{$\int_{\fp}$}  \mbox{$\int_{\vp}$}  \vp^*  \fq^*\cf
                                    &\eq\ 
                                    \mbox{$\int_{\fp}$} \Big(\big(\mbox{$\int_{\vp}$}  \vp^*
                                    \oo_{X\times T}\big)\,\o_{\oo_{X\times T}}\, \fq^*  \cf\Big)
                                  \label{f2}\\
&\eq
 \int_{\fp}\left(\big(\mbox{$\int_{\vp}$} \oo_Z\big)\,\o_{\oo_{X\times T}}\, \fq^*  \cf\right)
                                                \eq
\mbox{$\int_{\fp}$} (\fm\,\o_{\fq^\hdot\oo_T}\,\fq^\hdot\cf).\nonumber
\end{align}                                  
Combining  isomorphisms \eqref{f1} and \eqref{f2}  yields \eqref{iso fun}.

We now prove part (2) of the lemma.
    Let $\T^*_VU$ denote  (the total space of) the conormal bundle 
of  a smooth
subvariety $V$ of  a  smooth variety $U$. 
We
write $\xi_u\in\T^*_uU$ for a covector at a point $u\in U$.
In particular,
$\T^*_UU$ is the  zero section of ~$\T^*U$ and $0_u$  is the zero covector at ~$u$.

First, we claim that the  assumptions in (2)
imply an inclusion
\beq{nchar}
\SSS(\fm)\cap (\T^*_XX\times \T^*T)\ \sseteq\
\T^*_XX\times\T^*_TT\eq\T^*_{ {X\times T}}( {X\times T}).
\eeq
To see this, we
use the natural diagram
\[
\xymatrix{
\T^*Z\ &&\ Z\times_{X\times T}(\T^*X\times\T^*T)\  \ar[rr]^<>(0.5){pr_X\times pr_T}
\ar[ll]_<>(0.5){d^*{p}\,\times\,d^*{q}}
&& \  \T^*X\times \T^*T.
}
\]

The map ${p}\times {q}$ being proper,
we deduce $\SSS(\fm)\sseteq (pr_X\times pr_T)(d^*{p}\,\times\,d^*{q})\inv(\SSS(\oo_Z))$.
We compute
\begin{align*}
\SSS(\fm)\cap (\T^*_XX\times \T^*T)&\sseteq
\Big((pr_X\times pr_T)({d^*p}\,\times\,d^*{q})\inv(\T^*_ZZ)\Big)\ \bigcap\ 
(\T^*_XX\times \T^*T)\\
&
\eq\{(\xi_{\tfp(z)},\xi_{\tfq(z)})\in\T^*X\times\T^*T\mid
\xi_{{p}(z)}=0,\ d^*{q}(\xi_{\tfq(z)})=0,\ z\in Z\}\\
&\eq\{(0_{\tfp(z)},0_{\tfq(z)})\in\T^*X\times\T^*T,\ 
z\in Z\}\ \sseteq\  \T^*_XX\times \T^*_TT,
\end{align*}
where the last equality holds  since the assumption that ${q}$ be a smooth morphism
implies  that the map $d^*{q}$ is injective. This proves \eqref{nchar}.

To complete the proof of (2)  we must check that
\beq{tor}L^k(\Id_X\times i)^*\fm=0,\qquad \forall k\neq 0,
\eeq
for any imbedding  $i: V\into T$, of a closed subscheme.
By noetherian induction (aka `devissage'), this is equivalent to proving that \eqref{tor}
holds for any  imbedding $i: V\into T$, where $V$ is a {\em smooth}
 locally-closed subvariety of $T$. For any such $V$, 
the inclusion in \eqref{nchar}
implies that the $\dd_{X\times T}$-module $\fm$ 
 is noncharacteristic with respect to
the subvariety $X\times V\sseteq X\times T$.
It follows, \cite[Section 2.4]{HTT}, 
that the $\dd$-modules $L^k(\Id_X\times i)^*\fm$ vanish for all $k\neq 0$, as desired.
\end{proof}

\subsection{}
Let $\mm$ be the $(\ddg,\ddt)$-bimodule
that corresponds to the left $\dd(G\times T)$-module
$\Ga(G\times T, \cmm)$ via
the identification explained at the beginning of Section \ref{xyz sec}.
Taking global sections in \eqref{cmm} and writing $\langle-\rangle$ for 
$\C$-linear span, we find
\begin{align*}
  \mm&=\frac{\ddg\o\ddt}{(\ddg\ad\g) \o\ddt\,+\,
    (\ddg\o\ddt)\{u\o1-1\o \rad(u),\ u\in \dd(G)^G\}}\\
  &=
    (\nn\o\ddt)\big/\langle
    n u \o s- n\o \rad(u)s,\ n\in\nn,\, s\in \ddt,\, u\in {\dd(G)^G}^{}\rangle
    =\nn\o_{\ddt^W}\ddt.
\end{align*}

Thus,  with the above identification,
from Theorem \ref{pimm} 
we obtain an isomorphism
\beq{M}
R\Ga\big(G\times T,  \mbox{$\int_{p\times q}$}\,\oo_{\tg}\big)\ccong
\mm\,=\,
 \nn\o_{\ddt^W}\ddt.
\eeq

It is known
that the algebra $\ddt$ is flat as a $\ddtw$-module,
cf. Proposition \ref{eg}(ii) of Section \ref{d sub}.
Therefore, one has  isomorphisms of functors
\begin{align}
\nn\lo_{\ddt^W}(\dash)&\ccong\nn\,\lo_{\dd(T)^W}\,\ddt\,\lo_{\ddt}(\dash)\label{mno}\\
&
\ccong
\big(\nn\o_{\ddtw}\ddt\big)\,\lo_{\dd(T)}(\dash)
\ccong
\mm\,\lo_{\dd(T)}(\dash).\nonumber
\end{align}

The functors $R\Ga(G,\dash)$ and $R\Ga(T,\dash)$ being equivalences
of categories,
we  see that  Theorem \ref{main-thm} is equivalent to the following result.

\begin{thm}\label{mm-prop} \vi The following
functors  $\dcoh(\dd_T)\to\dcoh(\dd_G)$ are isomorphic:
\beq{gt-iso}
R\Ga(G, \mbox{$\int_{p}$}\,{q}^* (\dash))\ccong \mm\, \lo_{\ddt}\,R\Ga(T,\dash).
\eeq

\vii  The bimodule
$\mm$ is flat as a right $\ddt$-module.
\end{thm}

\begin{proof}[Proof of Theorem \ref{mm-prop}] 
We apply Lemma \ref{xyz} in the case where 
$X=G,\,  Z=\tg$, and the diagram above the statement of the lemma is
$G\xleftarrow{p}\tg\xrightarrow{q} T$.  In this case, by Theorem \ref{pimm}
we have $\fm=\int^0_{p\times q}\oo_\tg=\cmm$.
Thus, part (2) of Lemma \ref{xyz} implies
that $\cmm$ is a flat $q^\hdot\oo_T$-module.
Hence, the functor $\cmm\,\lo_{q^\hdot\oo_T} q^\hdot(\dash)$ is exact.

Observe next  that
any object in the essential image of the functor  $\cmm\,\lo_{q^\hdot\oo_T} q^\hdot(\dash)$
is
  supported on   $G\times_\fc T$,  since the support of 
$\cmm$ equals  $G\times_\fc T$.
The restriction of the first projection $\fp:  G \times T \to G$ to the subvariety $G\times_\fc T$
is a finite morphism. Therefore,  applying Lemma \ref{fin mor} below
we deduce
that the composite functor $\int_\fp(\cmm\lo_{\fq^\hdot\oo_T}  \fq^\hdot(\dash))$
is exact.
Thus,  it follows from the first isomorphism  \eqref{iso fun}  that the functor $\int_p q^*$ is exact,
which is part of the statement of Theorem \ref{main-thm}.

Finally, for any $\cf\in \dcoh(\dd_T)$, using the composite isomorphism in  \eqref{iso fun},
we obtain
\begin{align*}
  R\Ga(G,\, \mbox{$\int_p$}\, q^*\cf)&\ccong
                               R\Ga(G, \,R\fp_\idot(\cmm\lo_{\fq^*\dd_T} \fq^*\cf)\big)\\
  &\ccong
    R\Ga(G \times T,\, \cmm\lo_{\fq^*\dd_T} \fq^*\cf)\big)\\
  &\ccong
  R\Ga(G \times T, \cmm)\lo_{\ddt} R\Ga(T,\cf)=
  \mm\lo_{\ddt} R\Ga(T,\cf).
  \end{align*}
  This proves part (i) of Theorem \ref{mm-prop}.
  Part (ii)  follows from (i) since we have shown
  that the
    functor $\mm\lo_\ddt R\Ga(T,\dash)\cong
  R\Ga(G,\, \int_p q^*(\dash))$ is exact.
 \end{proof}

The lemma below is, of course, well known in the case of holonomic $\dd$-modules
where it is proved
using that  finite morphisms are proper, hence,  the corresponding
push-forward functor $\int$  commutes with
Verdier duality. This argument is not sufficient in the general case
since the abelian category
 of not necessarily 
 holonomic $\dd$-modules is not stable under Verdier duality.

 \begin{lem}\label{fin mor} Let $X$ be a smooth variety and $Y\subseteq X\times T$ a closed,
   not necessarily smooth, subvariety such that the
   first projection $\fp: Y\to X$ is a finite morphism.
Define a full subcategory
$D^b_{Y,\coh}(\dd_{\xxx})$ of $\dcoh(\dd_{\xxx})$ to be the category
formed by the objects $\cg$ such that $\supp(\H^i\cg)\subseteq Y$
for all $i$.
 Then, the functor $\int_\fp: D^b_{Y,\coh}(\dd_{\xxx})\to \dcoh(\dd_X)$ is exact.
\end{lem}
\begin{proof}
  Using induction on $\dim T$ we reduce the statement
to the case of  a 1-dimensional torus $T=\C^*$. Thus,
we may assume that $\C[T]=\C[t,t\inv]$,
so  $\partial:={t\frac{d}{dt}}$ is a translation invariant vector field on
$T$. Then, the object $\int_\fp\cg$
is represented by a two-term complex
$[\fp_\idot\cg\xrightarrow{\partial}\fp_\idot\cg]$ concentrated in degrees $-1$ and $0$, and
we have $\H^{-1}(\int_\fp\cg)=\Ker\partial$.
Thus, we must show that $\Ker\partial=0$. 

We may assume that $X$ is affine and work with global sections.
Let $u\in \Ker\partial$ be a nonzero  global section.
Let $I$ be the annihilator of $u$ in $\C[X\times T]$, resp.  $I_X$ the annihilator of $u$ in $\C[X]$. Clearly, one has
$\C[X\times T] I_X\subset I$.
Observe that the inclusion here is strict since
the  composite $\supp(u)\into\supp(\cg)\to X$ is a finite morphism.
Hence, there exists a function $f\in I\sminus (\C[X\times T] I_X)$ such that $fu=0$.  
Since $\partial$ kills $u$, it follows that
the operator $(\ad \partial)^N(f)$  kills $u$, for all $N\geq0$. 
We can write
$f$ as a finite sum $f=\sum_{j\in\Z}\ t^j \cdot f_j$,  for some $f_j\in \C[X]$. Since
$(\ad\partial)^N(t^j f_j)= j^N\cdot t^j f_j$, we deduce that
$\sum_j\ j^N\cdot t^j f_ju=0$, for all $N$.
 This implies 
that   each of the functions $f_j$ kills $u$, that is, $f_j\in I_X$.
We conclude that   $\sum_j\ t^j f_j\in \C[X\times T] I_X$, contradicting the choice of ~$f$.
\ep

\begin{rem}\label{indep} One can check that neither the proof of Theorem
  \ref{pimm}
  nor  the proof of Theorem
  \ref{mm-prop}  relies on the fact that the algebra map $\rad$ in 
  \eqref{ag} is actually an isomorphism. This fact is only used
  to insure
  an implication
  \[ \text{$\mm$ is flat over $\ddt$\en\ $\Longrightarrow$\en\
      $\nn$ is flat over $\ddtw$.}
    \]

    Thus, the proof of the exactness of parabolic induction does not
    rely on the difficult results of Levasseur and Stafford \cite{LS1}-\cite{LS2}.
  \end{rem}

\section{Parabolic restriction}
\label{pf-sec} 
\subsection{}\label{d sub} 
In this subsection we collect some 
results concerning $W$-equivariant $\dd(T)$-modules.

Let $K$ be  an  algebraic group. Given  a smooth variety $X$ with a $K$-action,
we denote by  $(\ddx,K)\mmod$, resp. $(\dd_X,K)\mmod$,
the abelian category of   {\em strongly} $K$-equivariant 
$\ddx$-modules, resp. $\dd_X$-modules.

In the rest of this  subsection  we will use simplified notation
$\dd:=\ddt$.
The Weyl group $W$ acts on  $\dd$ by algebra automorphisms.
It is clear that the category
$(\dd,W)\mmod $ is equivalent to the category 
$({W}\ltimes \dd)\mmod$  of modules
over ${W}\ltimes \dd$, a smash-product algebra.

One has the following standard result.

\begin{prop}\label{eg} \vi The algebras ${W}\ltimes \dd$ and $\dd^W$ are   Morita equivalent; specifically, one has
  the following  mutually inverse equivalences:
  \beq{mor}
  \xymatrix{
    ({W}\ltimes \dd)\mmod\  \ar@<0.4ex>[rrr]^<>(0.5){(\dash)^W}&&&
\
\dd^{W}\mmod.\ \ar@<0.4ex>[lll]^<>(0.5){\dd\o_{\dd^W }(\dash)}
}
\eeq

\vii The algebra $\dd$ is finitely generated and projective as a left, resp. right, $\dd^{W}$-module.

\viii There is an isomorphism
$
({W}\ltimes \dd)\cong \ddt\o_\ddtw\ddt$, of
$W\times W$-equivariant $\ddt$-bimodules.

\iv For any $L\in D^b(\dd^W\mmod)$, 
 let $\dd$ act on $R\Hom_{\dd^W }(\dd,L)$ via right multiplication on the domain $\dd$.
Then, in $D^b(\dd\mmod)$, there is a canonical isomorphism
\[R\Hom_{\dd^W }(\dd,L)\ccong \dd\lo_{\dd^W } L.\]
\end{prop}

\begin{proof}  Part (i) is well known. In more detail, let
 $e=\frac{1}{|{W}|}\sum_{w\in{W}}\ w$ 
be  the averaging idempotent 
of the group algebra $\C {W}$, and put  $H:={W}\ltimes \dd$.
It is immediate to check that inside $H$, we have
$\dd^{W}= eHe$. Further, using that the algebra $\dd$ is simple, one proves
an equality $HeH=H$. This implies the Morita equivalence.
Further, since $H$ is   projective as a left  $H$-module,
it follows from the Morita equivalence
that
$eH$ is  projective as a left  $eHe$-module, proving  (ii).
Part (iii) says that the map $He\o_{eHe} eH\to H$ induced by multiplication
is an isomorphism. The latter statement follows from the Morita equivalence
by multiplying both sides by $e$ on the left.

To prove (iv), we use 
the pairing 
$eH\o_{H} He\to eHe,\ eu\o ve\mto euve$.
This pairing is known to be perfect, that is,
it induces an isomorphism
$\dd\to\Hom_{\dd^{W}}(\dd,\dd^{W})$,
of right $\dd^{W}$-modules,  cf. 
\cite{EG}, Theorem 1.5(iii) for a more general result.
Since $\dd$ is a projective $\dd^W $-module, tensoring both sides of the
isomorphism with $L$, we obtain a chain isomorphisms
\[\dd\o_{\dd^W } L \ \iso\  R\Hom_{\dd^{W}}(\dd,\dd^{W})\,\lo_{\dd^W } L\ 
  \iso\ 
  R\Hom_{\dd^{W}}(\dd, L),
\]
and (iv) follows. \ep
\vskip3pt

      \begin{proof}[Proof of Corollary \ref{mla}] We will use simplified notation  ${S}=\sym\t$.
        For $w\in W$ and an $S$-module $V$, let $V^w$ denote an ${S}$-module obtained from $V$ by twisting
        the ${S}$-action via the dot-action of $w$, i.e. such that $t\in\t$ acts on $V^w$ by
        $v\mto (w\cdot t)v$. The assumption that  $\supp V=\{\la\}$ is a regular point of $\t^*$ implies
        an isomorphism ${S}\o_{{S}^W} V = \oplus_{w\in W}\, V^w$, of $W\ltimes {S}$-modules.
        We deduce the following isomorphisms of $W\ltimes\dd$-modules:
        \begin{align*}
          \dd\o_{{S}^W} V &\ccong\dd\o_{S}({S}\o_{{S}^W} V)
          \ccong 
                             \dd\o_{S}(\oplus_{w\in W}\, V^w)\\
                           &\ccong \oplus_{w\in W}\  \dd\o_{S}V^w\,\ccong\,
        (W\ltimes\dd)\,\o_{{S}}\, V\\
                                          &\stackrel{\star}\ccong (\dd\o_{\dd^W}\dd) \o_{S} V          \ccong                                             \dd\o_{\dd^W} (\dd\o_{S} V),
        \end{align*}
        where isomorphism $(\star)$ holds                  by Proposition \ref{eg}(iii).

      Taking $W$-invariants yields  an isomorphism
                                            $\dd^W\o_{{S}^W} V\cong \dd\o_{S} V$, of $\dd^W$-modules.
                                            The desired statement now follows from Theorem \ref{main-thm}
                                            and the following isomorphisms
                                            \[\nn\o_{\dd^W}(\dd\o_{S} V)\ccong
                                              \nn\o_{\dd^W}(\dd^W\o_{{S}^W} V)\ccong \nn\o_{S^W}V.
                    \qedhere                        \]
                          \end{proof}

        \subsection{}         The functor of {\em parabolic restriction} is defined as the functor
        $\int_q  p^*: D^b(\dd_G)\to D^b(\dd_T)$.
    The functor  $\int_p q^*(-)$ of parabolic induction  commutes with
 Verdier duality, since  the morphism $p: \tg\to G$ is proper
 and the morphism $q:\tg\to T$ is smooth. It follows that 
 the functor   of   parabolic restriction is a right adjoint of  the functor of parabolic induction.

\begin{prop} \label{M cor} There is an isomorphism of functors
 that makes
  the following diagram commute
\[
\xymatrix{
D^b(\dd_G\mmod)\ \ar@{=}[d]^<>(0.5){R\Ga(G,\dash)}
\ar[rrrrr]^<>(0.5){\int_q   p^*(\dash)}
&&&&&\ D^b(\dd_T\mmod),\ \ar@{=}[d]^<>(0.5){R\Ga(T,\dash)}\\
D^b(\ddg\mmod)\ \ar[rrrrr]^<>(0.5){\ddt\lo_{\ddtw}\,R\Hom_\ddg(\nn,\dash)}
&&&&& D^b(\ddt\mmod)
}
\]
\end{prop}
\begin{proof}
       Using a  standard adjunction between $R\Hom$ and $\lo$,
      we obtain the following isomorphisms of functors
\begin{align*}
  R\Hom_{\ddg}(\mm, \dash)&=R\Hom_{\ddg}(\nn\o_\ddtw\ddt, \dash)\\
  &\cong
       R \Hom_{\ddtw}\big(\ddt,\, R\Hom_{\ddg}(\nn,\dash)\big)\\
     &   \cong\ddt\lo_{\ddtw}\,R\Hom_{\ddg}(\nn,\dash),
\end{align*}
where the last  isomorphism follows from Proposition \ref{eg}(iv).
 Hence, for  $E\in D^b(\ddg\mmod)$ and $F\in D^b(\ddt\mmod)$, we find 
      \begin{align*}
    R\Hom_\ddg(\mm\lo_\ddt F,\, E) &=R\Hom_\ddt\big(F,\ R\Hom_\ddg\big(\mm, E)\big)\\
                                  &=R\Hom_\ddt\big(F,\ \ddt\lo_{\ddtw}\,R\Hom_{\ddg}(\nn,E)\big).
                                    \end{align*}

                                    We conclude that the functor $\ddt\lo_{\ddtw}\,R\Hom_{\ddg}(\nn,\dash)$ is a right adjoint of the functor  $\mm\lo_\ddt (\dash)$.
    The result now follows from   Theorem \ref{mm-prop}(i).
                                  \end{proof}
\vskip4pt
                                  
                                  It is  well known that parabolic restriction can be upgraded
                                  to a functor
                                  $\Res: D(\dd_G,G)\to D(\dd_T,W)$
between equivariant
derived categories.
We restrict our attention to   abelian  categories.
Let $\ind: \ddt\mmod\to (\dd_G,G)\mmod$, resp.
$\res_W: (\dd_G,G)\mmod\to (\dd_T,W)\mmod$,
be  the functor that  corresponds to
the functor  $\H^0\big(\mbox{$\int_p$}\,  q^*(\dash)\big)$,
resp. $\H^0\big(\mbox{$\int_q$}\,  p^*(\dash)\big)$,
via the equivalences $\Ga(G,\dash)$ and $\Ga(T,\dash)$.
Further, given a $\ddt$-module $F$ and $w\in W$, let
$F^w$ be the $\ddt$-module obtained by twisting the
$\ddt$-action by the automorphism $w$.
The  $\ddt$-module
$\oplus_{w\in W}\, F^w $
has the natural $W$-equivariant structure
such that the action of $y\in W$ is given by the
identity maps $F^w\to F^{yw}$.



                                    Our main result concerning parabolic restriction reads

                                    \begin{thm}\label{M prop} \vi The functor $\res_W$  is  isomorphic
                                      to the functor $E\mto \ddt\o_\ddtw E^G$.

 \vii The composite functor $\res_W\ccirc \ind: \ddt\mmod\to (\dd_T,W)\mmod$
 is isomorphic to the functor $F\mto \oplus_{w\in W}\, F^w$.
 
                                \viii      The functor $\res_W$ has  a left adjoint  functor
                                      \beq{ind} \ind_W:\  (\ddt,W)\mmod \to (\ddg,G),
                                      \en F\mto \nn\o_\ddtw F^W.
                                      \eeq

                                      The functors $\ind_W$ and $\res_W$ are exact; moreover,
                       the functor $\res_W$  induces an equivalence
 \beq{equiv}
   (\ddg,G)\mmod/\!\Ker(\res_W)\ \iso\ (\dd_T,W)\mmod.
   \eeq
                                    \end{thm}

\begin{rems}
\vi One can check that the $W$-action on $\ddt\o_\ddtw\Ga(G,\ce)$ induced by the natural $W$-action on
the first tensor factor  corresponds, via  isomorphism
of functors $\res$ and $ \ddt\o_\ddtw (\dash)^G$,
to the  $W$-action on $\Ga(T, \res_W\ce)$  that comes from the $W$-equivariant
structure on $\res_W \ce$
induced from the one on $\Res\ce \in D(\dd_T,W)$.
Conversely,
one can use the  $W$-action on $\ddt\o_\ddtw\Ga(G,\ce)$
to give $\H^0(\int_q  p^*\ce)$ a $W$-equivariant structure.
\vskip 2pt

\vii Part (ii) of the theorem was obtained  by T.H. Chen, \cite{Ch},
Proposition 3.2, by a different (less elementary)
method following an earlier result, \cite{Gu},  in the Lie algebra setting.

 \viii It follows from the theorem that an object $\ce$
    of $(\dd_G,G)\mmod$ is killed by the functor $\res_W$ iff  $\ce$ has no nonzero $G$-invariant
    global sections.
  \end{rems}
  
  \begin{proof}[Proof of Theorem \ref{M prop}]  
       For any $\ddg$-module $E$, one has
       \beq{eegg}
       \Hom_{\ddg}(\nn,E)=\Hom_{\ddg}(\ddg/\ddg\ad\g,E)=E^{\ad\g}=E^G.
\eeq
The isomorphism of functors stated in (i)
follows from this and  Proposition \ref{M cor}.

Observe next that for any  $\ddt$-module $F$
we have $(\mm\o_\ddt F)^G=\mm^G\o_\ddt F=F$.
Hence, from  part (i), using Proposition \ref{eg}(iii), we find
\begin{align*}
  \res_W(\ind F)&= 
                  \ddt\o_\ddtw (\mm\o_\ddt F)^G=\ddt\o_\ddtw F\\
  &=
 \big(\ddt\o_\ddtw \ddt\big)\o_\ddt F=\big(W\ltimes \ddt\big)\o_\ddt F.
 \end{align*}
The last expression is isomorphic to $\,\oplus_{w\in W} \, F^w$, proving (ii).

The functor  $\ind$
is exact since $\nn$ is a flat $\ddtw$-module.
    The functor $(\ddg,G)\mmod\to \ddtw\mmod,  E\mto E^G$,
    is exact since the $G$-action on  $G$-equivariant $\ddg$-modules is semisimple.
    Since
    $\ddt$ is   a flat $\ddtw$-module, 
    we deduce that  the functor $E \mto \ddt\o_\ddtw E^G$, hence the functor $\res_W$,  is
    exact.
    
    To prove that $\ind_W$ is a left adjoint of $\res_W$, observe
      that by  Morita equivalence,
       for any $F,F'\in (\ddt,W)\mmod $, 
                                    the natural map
                                    $\Hom_{(\ddt,W)\mmod }(F,F')\to
                                    \Hom_\ddtw(F^W, (F')^W)$ is an isomorphism.
                                    Hence, for  any $E\in (\ddg,G)\mmod$, using \eqref{eegg}
                                    we obtain
                                    \begin{align*}
                                      \Hom_{(\ddt,W)\mmod }(F,\,\res_W E) &\eq
                                                                          \Hom_{(\ddt,W)\mmod }(F,\,\ddt\o_{\ddtw} E^G)\\
                                      &\eq \Hom_\ddtw(F^W, \,\Hom_{\ddg}(\nn,E))\\
                                      & \eq\Hom_\ddg(\nn\o_\ddtw F^W,E)=\Hom_\ddg(\ind_W F,E).
                                    \end{align*}

      It remains to show that \eqref{equiv} is an equivalence.
               It is sufficient to check that for any $W$-equivariant $\ddt$-module
                         $F$             the adjunction morphism $F\to \res_W\ccirc\ind_W(F)$
                                      is an isomorphism.      To this end, we compute
                                    \[
                                      (\nn\o_\ddtw F^W)^G=
                                                                        (\nn^G)\o_\ddtw F^W =\ddtw\o_\ddtw F^W=F^W.
                                    \]
                                    Thus, using Proposition \ref{eg}(i) we obtain
                                    \[
                                      \res_W(\ind_W F)=\ddt\o_{\ddtw} \big((\nn\o_\ddtw F^W)^G\big)\eq
                                      \ddt\o_{\ddtw} F^W=F. \qedhere
                                    \]                  
                               \end{proof}

Let $\mm^t$ be a $(\ddt,\ddg)$-bimodule obtained from the
 $(\ddg,\ddt)$-bimodule $\mm$ by replacing the left $\ddg$-action
 by a right $\ddg$-action, resp. the right $\ddt$-action by a left
 $\ddt$-action, using the invariant volume form $dg$, resp. $dt$.

    \begin{cor}\label{res-tens}
   \vi    The functor $\res_W$ is isomorphic to the functor
   $E\mto \mm^t\o_\ddg E$.

   \vii    The functor $\res_W$ takes finitely generated $\ddg$-modules
   to  finitely generated $\ddt$-modules.
    \end{cor}

    \begin{proof}
      For any $\ddg$-module $E$, we have
      \[\g E\backslash E\ccong  (\ad\g\,\ddg\backslash\ddg)\o_\ddg E.
      \]

      If $E$ is, in addition,  a  semisimple  as a $G$-representation. Then,
the composite  $E^{\ad\g}\into E\onto
    \g E\backslash E$ is an isomorphism.
    Thus, using Theorem \ref{M prop}(i) we find
    \[
      \res_W(E)\ccong \ddt\o_\ddtw E^G\ccong \ddt\o_\ddtw (\ad\g\,\ddg\backslash\ddg)\o_\ddg E
      \ccong \mm^t\o_\ddg E.
      \]

      This proves (i).
      Part (ii) is a consequence of a well known result due to Hilbert;
      specifically, the result implies
      that $E^G$ is  finitely generated
      over $\ddg^G$ provided $E$ is  finitely generated over
      $\ddg$.
    \end{proof}

   \begin{rem} S. Gunningham  \cite{Gu} proved that
        Lie algebra  analogues $D(\dd_\g,G)\leftrightarrows D(\dd_\t,W)$
        of the functors
        $\Ind$ and $\Res$ preserve coherence.
        \end{rem}
\small{
\bibliographystyle{plain}

}

\end{document}